\documentclass[12pt]{amsart}
\usepackage{amsmath, amssymb,amscd, amsthm , amsmath}

\usepackage{geometry}

\newtheorem{theorem}{Theorem}
\newtheorem{lemma}[theorem]{Lemma}
\newtheorem{corollary}[theorem]{Corollary}
\newtheorem{proposition}[theorem]{Proposition}

\title[GCD sums and applications]{Refinements of G\'al's theorem and applications}

\author{Mark Lewko and Maksym Radziwi\l\l}
\address{Department of Mathematics, UCLA, Los Angeles CA 90095--1555, USA}
\email{mlewko@gmail.com}
\address{School of Mathematics, Institute for Advanced Study, 1 Einstein Drive,
Princeton, NJ, USA}
\email{maksym@math.ias.edu}
\thanks{The first author is supported by a NSF postdoctoral fellowship,
DMS-12042 and the IAS Fund for Mathematics, the second author was
partially supported by NSF grant DMS-1001068}

\date{\today}

\begin{document}

\begin{abstract}We give a simple proof of a well-known theorem of G\'al and of the recent related results of Aistleitner, Berkes and Seip \cite{ABS} regarding the size of GCD sums. In fact, our method obtains the asymptotically sharp constant in G\'al's theorem, which is new. Our approach also gives a transparent explanation of the relationship between the maximal size of the Riemann zeta function on vertical lines and bounds on GCD sums; a point which was previously unclear. Furthermore we obtain sharp bounds on the spectral norm of GCD matrices which settles a question raised in \cite{ABSW}.
We use bounds for the spectral norm to show that series
formed out of dilates of periodic functions of bounded variation converge
almost everywhere if the coefficients of the series are in $L^2 (\log\log 1/L)^{\gamma}$, with
$\gamma > 2$. This was previously known with $\gamma >4$, and is known to fail for $\gamma<2$. We also develop a sharp Carleson-Hunt-type theorem for functions of
bounded variations which settles another question raised in \cite{ABS}.
 Finally we obtain
almost sure bounds for partial sums of dilates of periodic functions of
bounded variations improving \cite{ABS}. This implies almost sure bounds for the discrepancy of $\{n_k x\}$ with $n_k$ an arbitrary
growing sequences of integers.

\end{abstract}


\maketitle
\section{Introduction}\label{sec:intro}

Let $1 \leq n_1 < n_2 < \ldots < n_k$ be an arbitrary sequence of integers.
The problem of bounding GCD sums
\begin{equation} \label{GCDSum}
\frac{1}{k} \sum_{i,j \leq k} \frac{(n_i, n_j)^{2\alpha}}{(n_i n_j)^{\alpha}} \ \ ,  \ \
0 < \alpha \leq 1
\end{equation}
arises naturally in metric Diophantine approximation, following
Koksma's initial work \cite{Koksma} (see also \cite{DyerHarman}).

While estimating (\ref{GCDSum})
for many specific sequences is straightforward,
the problem of determining the maximal size of (\ref{GCDSum}) among all sequences with $k$ terms is much more subtle and
the case $\alpha = 1$ was posed as a prize
problem by the Scientific
Society at Amsterdam in 1947 on Erd\"os's suggestion.

The problem for $\alpha = 1$ was solved by I. S. G\'al in 1949 \cite{Gal}.
G\'al's proof
is a difficult combinatorial argument spanning 20 pages. G\'al
showed that for $\alpha = 1$
the GCD sum (\ref{GCDSum}) is bounded by $C (\log\log k)^2$ with $C > 0$ an absolute constant and moreover that this bound is optimal up to the value of the constant $C > 0$.

Our first contribution is a simple, two page proof, of G\'al's theorem. In addition our proof
determines
the optimal constant $C$ as $k \rightarrow \infty$, which is new.

\begin{theorem}\label{thm:asy} As $k \rightarrow \infty$,
$$
\sup_{1 < n_1 < \ldots < n_k} \frac{1}{k (\log\log k)^2} \sum_{i,j \leq k}
\frac{(n_i, n_j)^2}{n_i n_j} \longrightarrow \frac{6 e^{2\gamma}}{\pi^2}.
$$

\end{theorem}

The extremal sequence in Theorem \ref{thm:asy} is supported on very smooth integers.
The generalization of (\ref{GCDSum}) to $\tfrac 12 \leq \alpha < 1$ was studied by Dyer and Harman \cite{DyerHarman} who were also interested in applications to metric diophantine
approximations.
Recently, Aistleitner, Berkes and Seip \cite{ABS} showed that for
$1/2 < \alpha < 1$, the GCD sum (\ref{GCDSum}) is bounded by
\begin{equation} \label{Sum1}
\ll \exp \Big ( C(\alpha) \cdot \frac{(\log k)^{1 - \alpha}}{(\log\log k)^{\alpha}} \Big ).
\end{equation}
This bound is sharp up to the value of the constant $C(\alpha)$.
Several authors have remarked on the similarity between this estimate and the conjectured maximal size of the zeta function along a vertical line (see \cite{Youness, Montgomery})
which states, for $1/2 < \alpha < 1$,
\begin{equation} \label{Sum2}
\sup_{|t| \leq k} \log |\zeta(\alpha + it)| \asymp
\frac{(\log k)^{1 - \alpha}}{(\log \log k)^{\alpha}}.
\end{equation}

Our method generalizes to the case $\tfrac 12 < \alpha < 1$
and allows us to rederive in a simple way the results of
Aistleitner, Berkes and Seip \cite{ABS}.
In addition, the proof we give shows that the
GCD sum (\ref{GCDSum}) is essentially majorized by the square
%of the
%almost sure supremum
of the supremum
of a certain random model of the zeta function, which is in fact
used to conjecture (\ref{Sum2}) (see \cite{Youness, Montgomery} and Section 9 at the end of this paper). In particular the similarity between
(\ref{Sum1}) and (\ref{Sum2}) is not a coincidence.

Our method also generalizes to the \textit{spectral norm} case, which is a key ingredient in the applications described later on. In particular, if we let $\mathbf{c} = (c(1), \ldots, c(n))$, then one is interested in bounding the quantity
$$
\sup_{\| \mathbf{c} \|_2^2 = 1} \frac{1}{k} \sum_{i,j \leq k}
\frac{(n_i, n_j)^{2\alpha}}{(n_i n_j)^{\alpha}} \cdot c(n_i)
\overline{c(n_j)}.
$$
Sharp bounds for
the spectral norm have been established for $1/2 < \alpha < 1$ in
\cite[Theorem 5]{ABS}  but
the case $\alpha = 1$ had remained open (see \cite{ABS}, \cite{ABSW}). Our main theorem settles this problem.

%Besides, our proof is much simpler, contains Gal's theorem (the proof of
%Aisentdler, Berkes and Seip is unable to recover Gal's theorem),
%and yields sharp constants in Gal's theorem.

%allows us to provide a unified treatement of the case
%$\alpha = 1$, yields
%Besides, as we already our proof yields easily asymptotically sharp estimates for the constant in Gal's theorem,
%which until now seemed out of reach.

\begin{theorem}\label{thm:spec}
Let $\mathbf{c} = (c(1), \ldots, c(n))$. Then, for $\alpha = 1$,
\begin{equation} \label{Spectral}
\sup_{\| \mathbf{c} \|_2^2 = 1} \frac{1}{k} \sum_{i,j \leq k}
\frac{(n_i, n_j)^{2\alpha}}{(n_i n_j)^{\alpha}} \cdot c(n_i) \overline{c(n_j)}
\leq \Big ( \frac{6 e^{2\gamma}}{\pi^2} + o(1) \Big )
\cdot (\log\log k)^2.
\end{equation}
In addition for $0 < \alpha < 1$ we re-derive the bounds obtained
for the spectral norm by Aistleitner, Berkes and Seip in \cite[Theorem 5]{ABS}. Precisely, for $\tfrac 12 < \alpha < 1$ we bound the spectral norm by
$$
\exp \Big ( 2C(\alpha) \frac{(\log k)^{1 - \alpha}}{(\log\log k)^{\alpha}} \Big )
$$
and for $0 < \alpha \leq \tfrac 12$ by
$$
k^{1 - 2\alpha} \cdot \exp \Big ( 2C(\alpha) \sqrt{\log k \log\log k}
\Big )
$$
with $C(\alpha)$ an absolute constants depending only on $\alpha$.
\end{theorem}
%\begin{remark}
%With more work one can show that for $1/2 < \alpha < 1$ we have
%$C(\alpha) = 2 G_1(\alpha) \alpha^{-2\alpha} (1 - \alpha)^{\alpha - 1} + o(1)$
%as $k \rightarrow \infty$,
%with $G_1(\alpha) = \int_{0}^{\infty} \log I_0(u) u^{-1-1/\sigma} du$
%and $I_0(u) = \sum_{n = 0}^{\infty} (u/2)^{2n}/ n!^2$ the modified Bessel funct%ion of order 0.
%\end{remark}
The bound (\ref{Spectral}) answers a question raised by
Aistleitner, Berkes, Seip and Weber \cite{ABSW}
regarding the correct power of $\log\log k$ for
the spectral norm at $\alpha = 1$,
where it is described as a
 ``\textit{a profound problem}'' (see remarks after (33) in \cite{ABSW}).
% to decide whether
%the exponent 4 of $\log\log k$ is optimal}''.
We note that inequality (\ref{Spectral}) also
has an asymptotically sharp constant. For example, Hilberdink \cite{Hilberdink}
showed
that when $n_i = i$ the bound (\ref{Spectral}) is attained
as $k \rightarrow \infty$
for a certain choice of the coefficients $c(k)$
(this however is not true for $1/2 < \alpha < 1$,
see \cite{ABSW}). The bound is also attained when $n_k$ is
choosen to be the extremal sequence in Theorem \ref{thm:asy} and $c(k) = k^{-1/2}$.

The main idea in the proof of Theorem \ref{thm:spec} is that one can write (\ref{GCDSum}) as an integral involving the Riemann zeta-function (or, more precisely, a random model for the Riemann zeta-function) and then appeal to known distributional estimates for this quantity.

We note that while our method recovers completely Theorem \ref{thm:asy}
from \cite{ABS} and goes beyond when $\alpha = 1$,
the bound for $\alpha = 1/2$ is nonetheless not optimal.
It has been conjectured in \cite{ABSW}, correcting an older conjecture of Harman \cite{HarmanBook}, that the optimal bound
in the limiting case $\alpha = 1/2$ is
$$
\exp \Bigg ( C \cdot \sqrt{\frac{\log k}{\log\log k}} \Bigg ).
$$
The best results towards this
conjecture are due to Bondarenko and
Seip (see \cite{BondarenkoSeip}), who come within a triple logarithm
of the conjecture.

We now focus on applications of Theorem \ref{thm:spec}, in the spirit of \cite{ABS}.
Let $f$ be a function such that
\begin{equation} \label{conditions}
f(x + 1) = f(x) \text{ , } \int_{0}^{1} f(x) dx = 0.
\end{equation}
We are interested in $L^2$ conditions for the almost everywhere convergence
of
\begin{equation} \label{dilated}
\sum_{\ell = 1}^{\infty} c_{\ell} f(n_{\ell} x)
\end{equation}
for
$1 \leq n_1 < n_2 < \ldots$ an arbitrary sequence of integers, and in almost
sure bounds for
\begin{equation} \label{partsumbounds}
\sum_{i \leq k} f(n_i x).
\end{equation}
From the point of view of applications to metrical diophantine approximation
a natural choice of $f$ is $f(x) = \{x \} - \tfrac 12$, a function of
bounded variation (which implies that the $j$-th Fourier coefficient
of $f$ is $O(1/j)$). Choosing $f(x) = \chi_{I}(\{x\}) - |I|$ in (\ref{partsumbounds}), with $\chi_I$ the
indicator function of an interval $I \subset [0,1]$ relates the problem of
bounding (\ref{partsumbounds}) to that of obtaining almost sure bounds
for the discrepancy of $\{n_k x\}$.

For smooth functions such as $f(x) = \sin 2\pi x$, one can use a deep
result of Carleson \cite{Carleson} to show that if $c_{k}$ is in $\ell^2$
then (\ref{dilated}) converges almost everywhere. However already for  $f(x) = \{x\} - \tfrac 12$ the condition $c_k \in \ell^2$ is insufficient. In \cite[Section 6]{ABS} it is shown that for $f(x) = \{x\} - \tfrac 12$ and any $\gamma < 2$
there exists an increasing sequence of positive integers $n_i$ and a sequence of real numbers $c_{\ell}$, with
\begin{equation} \label{condition}
\sum_{\ell = 1}^{\infty} c_{\ell}^2 \cdot (\log\log \ell)^{\gamma} < \infty
\end{equation}
for which (\ref{dilated}) diverges almost everywhere.

Aistleitner, Berkes and Seip \cite[Theorem 3]{ABS} complemented this negative result by showing
that if (\ref{condition}) holds for $\gamma > 4$ then (\ref{dilated})
converges almost everywhere for any $f$ of bounded variation (and thus also
$f(x) = \{ x \} - \tfrac 12$). In the theorem below we close the
remaining gap.
\begin{corollary}\label{cor:AE}
Let $f$ be a function of bounded variation satisfying (\ref{conditions}) .
Let $c_k$ be a sequence of real numbers such that
$$
\sum_{\ell \geq 3} c_{\ell}^2 \cdot (\log\log \ell)^{\gamma} < \infty
$$
for some $\gamma > 2$. Then for every increasing sequence $(n_{\ell})_{\ell \geq 1}$ the series
$$
\sum_{\ell \geq 1} c_{\ell} f(n_{\ell} x)
$$
converges almost everywhere.
\end{corollary}
This result is optimal in the sense that the exponent $\gamma$ cannot be lowered any further. We note that our coundition
is roughly equivalent to $c_{\ell} \in L^2 \cdot (\log\log 1/L)^{\gamma}$, $\gamma > 2$, since the series composed of integers with $|c_{\ell}| < 1/\ell^2$ converges absolutely.
Our method of proof also allows us to recover the recent
results in \cite{ABSW} obtained for functions $f$ with Fourier coefficients
decaying at a rate of $j^{-\alpha}$ with $1/2 < \alpha < 1$. Corollary \ref{cor:AE} also improves a very recent result of Weber \cite{Weber} where the same conclusion is obtained with a $(\log\log k)^{4} / (\log\log\log k)^2$ in place of $(\log\log k)^{2 + \varepsilon}$ for the special case $n_k = k$.

We also obtain the following improvement of a result in \cite[Theorem 2]{ABS}.
\begin{corollary}\label{cor:EK}
Let $f$ be a function of bounded variation satisfying (\ref{conditions}).
Then, for almost every $x$,
$$
\sum_{\ell \leq N} f(n_{\ell} x) \ll \sqrt{N \log N} (\log\log N)^{3/2 + \varepsilon}.
$$
\end{corollary}
This improves the exponent $\tfrac 52$ obtained in \cite[Theorem 2]{ABS} to $\tfrac 32$. The optimal exponent is conjectured to be $\tfrac 12$.
The problem of obtaining almost everywhere bounds on the quantity $\sum_{\ell \leq N} f(n_{\ell} x)$ has a long history, partly motivated by the problem of
obtaining almost sure bounds for the discrepancy of the sequence $\{n_k x\}$.
Indeed weaker estimates on this quantity (or special cases thereof) were obtained by G\'al \cite{Gal} (1949), Erd\"os and Koksma \cite{ErdosKoksma} (1949), G\'al and Koksma \cite{GalKoksma} (1950), Cassels \cite{Cassels} (1950), R. C. Baker \cite{Baker} (1981), Aistleitner, Mayer and Ziegler (2010), and Aistleitner, Berkes and Seip \cite{ABS} (2012). See also \cite{ABSW}, \cite{Berkes}, and \cite{Weber}.

% It seems that to go beyond Theorem 4
%one would need to show that the extremal examples
%in Gal's theorem and in Theorem 5 below are in a sense ``far apart''.

The key estimate in the proofs of Corollary \ref{cor:AE} and Corollary \ref{cor:EK} is an optimal Carleson-Hunt-type inequality for systems of dilated functions $\{ f(n_{\ell} x) \}$ with $f$ of bounded variation. The theorem below answers a question
in \cite{ABS} regarding the optimal version of the Carleson-Hunt theorem,
in this setting (see remarks after Lemma 4 in \cite{ABS}).
It would be fitting to call this a maximal analogue of G\'al's theorem.

\begin{theorem}\label{thm:BVmax}Let $f: \mathbb{T} \rightarrow \mathbb{C}$ be a complex-valued function on the circle with Fourier coefficients, $a(j)$, satisfying the decay condition $|a(j)| = O(|j|^{-1})$. Let $n_1,n_2,\ldots,n_N$ be a strictly increasing sequence of positive integers and $c(k)$ a sequence of complex numbers. Then,
$$ \int_{\mathbb{T}} \left( \max_{1\leq M \leq N} \left| \sum_{k=1}^{M} c(k) f(n_k x) \right|  \right)^2 \ll   \left( \log\log  N  \right)^2 \sum_{k=1}^{N} |c(k)|^2.$$
\end{theorem}
This inequality with
an exponent of $4$ instead of $2$ was obtained in \cite[Lemma 4]{ABS}, it's
also shown there that the function $(\log\log N)^2$ cannot be
replaced by any slower growing function.

There are several innovations compared to the proof in \cite{ABS}.
First, we perform a splitting according to the largest prime divisor,
secondly, we use a majorant principle to handle the tails composed
of large primes, and finally
we use the ideas that enter in our proof of Theorem \ref{thm:spec} to handle the
contribution of the large primes after the application
of the majorant principle.

Because of the splitting, which is done according to the largest
prime factor, it is tempting to investigate if there are any links
with the P-summation method of Fouvry and Tenenbaum,
applied to trigonometric series as in
de la Bret\`eche and Tenenbaum's paper \cite{Tenenbaum}.

\section{Notation}
We use the usual asymptotic notation. For instance, we write $X\ll Y$ to indicate that there exists a universal constant $C$ such that $|X| \leq C |Y|$. We let $\mathbb{T}$ denote the unit circle, and $e(x) := e^{2 \pi i x}$ for $x \in \mathbb{T}$. For $f \in L^1(\mathbb{T})$ we define the $j$-th Fourier coefficient by the relation
$$c(j) := \int_{\mathbb{T}} f(x) e(-jx) dx.$$
We let $\zeta(s)$ denote the Riemann zeta function.

\section{The random model}

Let $X(p)$ be a sequence of independent random variable, one for each prime $p$, and equidistributed
on the torus $\mathbb{T}$. For an integer $n$ we let
$$
X(n) := \prod_{p^{\alpha} \| n} X(p)^{\alpha}.
$$
The random model of the zeta-function that we will be working with is
the following
$$
\zeta(\sigma, X) := \prod_{p} \Big (1 - \frac{X(p)}{p^{\sigma}} \Big )^{-1}.
$$
Note that the product is convergent almost surely for $\sigma > \tfrac 12$ by Kolmogorov's
three series theorem. Note also that in an $L^p$ sense (with $p > 0$) for
$\tfrac 12 < \sigma < 1$,
$$
\zeta(\sigma, X) = \sum_{n} \frac{X(n)}{n^{\sigma}}
$$
and that
$$
\mathbb{E}[X(n)\overline{X(m)}] = \begin{cases}
1 & \text{ if } n = m \\
0 & \text{ otherwise}
\end{cases}.
$$
Instead of working with the probabilistic model we could also work
with the zeta function itself,  since for example,
$$
\lim_{T \rightarrow \infty} \frac{1}{2T} \int_{-T}^{T}
|\zeta(\sigma + it)|^2 \cdot n^{-it} dt = \mathbb{E}\Big [|\zeta(\sigma, X)|^2 \cdot \overline{X(n)}\Big].
$$
To re-inforce this point, the distributional estimates that we use in Lemma 7 below
are known unconditionally for $\zeta(\sigma + it)$ (see \cite{Youness, GranvilleSound, Sound}). These results however
 are
often obtained by first passing to the random model, for this reason
we did not see the advantage of working with $\zeta(\sigma + it)$ directly
which
relies on deeper machinery (for example the zeros of $\zeta(s)$ enter the analysis).
The random model described above is commonly used in the study of the
Riemann zeta-function, we refer the reader to \cite{Youness, GranvilleSound, Sound, LLR, Erdos}) for more information and examples of its application.

\section{The distributional estimate}

The lemma below is adapted from \cite[Lemma 2.1]{Youness}.

\begin{lemma}
We have the following bound,
$$
\log \mathbb{E} |\zeta(\alpha, X)|^{2\ell}
\leq \begin{cases}
2\ell(\log\log \ell+ \gamma + O((\log \ell)^{-1}) & \text{ for } \alpha = 1 \\
C(\alpha) \ell^{1/\alpha} (\log \ell)^{-1} & \text{ for } \tfrac 12 < \alpha < 1 \\
C(\tfrac 12) \ell^2 \log (\alpha - \tfrac 12)^{-1/2} & \text{ for } \alpha \rightarrow 1/2 \\
\end{cases}.
$$
\end{lemma}
\begin{proof}
Note that
$$
\mathbb{E}[|\zeta(\alpha, X)|^{2\ell} = \prod_{p} E_\ell(p)
\text{ with }
E_\ell(p) = \mathbb{E} \Big [ \Big | \Big ( 1 - \frac{X(p)}{p^{\alpha}}
\Big )^{-2\ell} \Big | \Big ].
$$
For $p < (2\ell)^{1/\alpha}$ we have the trivial bound $E_\ell(p) \leq (1 - 1/p^{\alpha})^{-2\ell}$. For $p > (2\ell)^{1/\alpha}$ we notice that
$$
E_\ell(p) = \frac{1}{2\pi} \int_{-\pi}^{\pi} \Big (1 - \frac{e^{i\theta}}{ p^{\alpha}} \Big )^{-\ell} \cdot
\Big (1 - \frac{e^{-i\theta}}{p^{\alpha}} \Big )^{-\ell} d\theta = I_0(2\ell/p^{\alpha})
(1 + O(\ell/p^{2\alpha}))
$$
where $I_0(z)$ is the 0-th modified Bessel function, and
$I_0(t) = \sum_{n} (t/2)^{2n} / (n!)^2$. In particular
$\log I_0(t) \ll t^2$ for $0 < t \leq 1$. Combining these bounds we get
$$
\log \mathbb{E}[|\zeta(\alpha, X)|^{2\ell}] \leq 2\ell \sum_{p < (2\ell)^{1/\alpha}}
\log \Big ( 1 - \frac{1}{p^{\alpha}} \Big ) +  C \sum_{p > (2\ell)^{1/\alpha}}
\frac{\ell^2}{p^{2\alpha}}.
$$
When $\alpha = 1$ the first sum contributes $2 \ell ( \log\log \ell +  \gamma + O(1/\log \ell))$ while the second
contributes $O(1)$. When $1/2 < \alpha < 1$ the prime number theorem shows that
the above sum is
$$
\ll \frac{\ell^{1/\alpha}}{\log \ell} \cdot \Big ( \frac{\alpha}{1-\alpha}
+ \frac{\alpha}{2\alpha - 1} \Big ).
$$
Finally when $\alpha$ tends to $1/2$ we use the more careful bound
$$
\sum_{p} \frac{\ell^2}{p^{2\alpha}} \ll \ell^2 \cdot \log (\alpha - \tfrac 12)^{-1}
$$
which comes from $\zeta(2\alpha) = 1/(2\alpha - 1) + O(1)$, to conclude.
\end{proof}

\section{Proof of Theorem \ref{thm:spec}}

Let $\mathcal{N}=\{n_1, \ldots, n_k\}$.
Let
$$
D(X) := \sum_{n \in \mathcal{N}} c(n) X(n).
$$
Consider the expression
$$
\mathbb{E} \Big [ |\zeta(\alpha, X)|^2 \cdot |D(X)|^2 \Big ].
$$
On the one hand, expanding the square this is equal to
\begin{align} \label{First} \nonumber
\sum_{n = 1}^{\infty} \Big | \sum_{\substack{k | n \\ k \in \mathcal{N}}}
c(k) \cdot \frac{k^{\alpha}}{n^{\alpha}} \Big |^2
& = \sum_{m,n \in \mathcal{N}} c(m) \overline{c(n)} \cdot \frac{(k m)^{\alpha}}
{[k,m]^{2\alpha}} \cdot \zeta(2\alpha) \\
& = \zeta(2\alpha) \sum_{i,j \leq k} \frac{(n_i, n_j)^{2\alpha}}{(n_i n_j)^{\alpha}}
\cdot c(n_i) \overline{c(n_j)}.
\end{align}
On the other hand, for any $\ell, V > 0$, and $\tfrac 12 < \alpha \leq 1$,
\begin{equation} \label{Secondeq}
|\zeta(\alpha,X)D(X)|^2 \leq e^{2V} \cdot |D(X)|^2
+ k \cdot |\zeta(\alpha,X)|^{2(\ell + 1)} \cdot e^{-2\ell V}.
\end{equation}
Indeed to prove this inequality note that if $|\zeta(\alpha, X)| < e^{V}$ then
the left-hand side is less than $e^{2V} \cdot |D(X)|^2$, while
if $|\zeta(\alpha, X)| > e^{V}$, then the left-hand
side is less than $$
|D(X)|^2 \cdot |\zeta(\alpha, X)|^{2(\ell+ 1)} \cdot e^{-2\ell V}
\leq k \cdot |\zeta(\alpha, X)|^{2(\ell + 1)} \cdot e^{-2\ell V},$$
using the
$L^{\infty}$ bound,
 $|D(X)|^2 \leq k \cdot \| \mathbf{c} \| \leq k$ coming from Cauchy-Schwarz.
Taking the expectation on both sides and using (\ref{First}) we get
\begin{equation} \label{Second}
\zeta(2\alpha) \cdot \frac{1}{k} \sum_{i, j \leq k} \frac{(n_i, n_j)^{2\alpha}}{(n_i n_j)^{\alpha}}
\leq
e^{2V} \cdot \| \mathbf{c} \|_2^2 + k \cdot \mathbb{E} [|\zeta(\alpha, X)|^{2\ell + 2}] \cdot e^{-2\ell V}.
\end{equation}
In the above equation, if $\alpha = 1$ then we let
$$
V = \log \log\log k + \gamma + 2/\psi(k) \text{ and }
\ell = \psi(k) \cdot \log k
$$
with $\psi(k) \rightarrow \infty$ very slowly as $k \rightarrow \infty$
(say $\psi(k) = \log\log\log k$). Otherwise, we let
\begin{align*}
V = \begin{cases}
(C(\alpha) \log k)^{1 - \alpha} \cdot (\log\log k)^{-\alpha} &  
\text{ if } \tfrac 12 < \alpha < 1 \\
(C(\tfrac 12) \log k \log\log k)^{1/2}, & \text{ if } \alpha = \tfrac 12 +
\tfrac{1}{\log k}
\end{cases} \text{ with }  \ell = \frac{\log k}{V}.
\end{align*}
With this choice of parameters we have
$\mathbb{E} [|\zeta(\alpha, X)|^{2\ell + 2} ] \cdot e^{-2\ell V} \ll k^{-1}$ for
a fixed $1/2 < \alpha \leq 1$ and for $\alpha = 1/2 + 1/\log k$. 
If $1/2 < \alpha \leq 1$ is fixed, then inserting the choice of
$\ell$ and $V$ made above into (\ref{Second}) gives the claim. 
In order to prove the claim for $\alpha = \tfrac 12$, we use Holder's inequality,
$$
\frac{1}{k} \sum_{i,j \leq k} \frac{(n_i,n_j)}{\sqrt{n_i n_j}}
\leq \Big ( \frac{1}{k} \sum_{i,j \leq k} \frac{(n_i,n_j)^{2\alpha}}{(n_i n_j)^{\alpha}}
\Big )^{1/(2\alpha)} \cdot k^{1 - 1/(2\alpha)}
$$
with $\alpha = 1/2 + 1/\log k$, and appeal to (\ref{Second})
with the choice of parameters as described above.
The result for $0 < \alpha \leq \tfrac 12$ follows
 in the same manner by interpolating with the case $\alpha = \tfrac 12$ using
Holder's inequality.

\section{Proof of Corollary 1}
We have already established the upper bound in Theorem \ref{thm:asy}.
Therefore it suffices to obtain the lower bound.
Let $\mathcal{P}(r,\ell) = p_1^{\ell - 1} \cdot \ldots \cdot p_r^{\ell - 1}$
where $p_1, p_2, \ldots$ are consecutive primes.
G\'al proves the following identity in \cite{Gal},
$$
\sum_{\substack{n_i, n_j | \mathcal{P}(r,\ell)}}
\frac{(n_i, n_j)^2}{n_i n_j}  = \prod_{p | p_1 \ldots p_r}
\Big ( \ell + 2 \sum_{v = 1}^{\ell - 1} \frac{\ell - v}{p^{v}}
\Big )
$$
where the summation goes over all $n_i$ and $n_j$ dividing
$\mathcal{P}(r,\ell)$\footnote{This identity can be also quickly checked using the
fact that $f(m,n) = (m,n)^2 / (mn)$ is a multiplicative
function of two variables}. The number of divisors of $\mathcal{P}(r,\ell)$ is $ \ell^{r}$.
 Let $r$ be the largest $r$ such that $(r + \log k)^r < k$.
Therefore
$r \sim \log k / \log\log k$ and $p_r \sim \log k$ by the prime
number theorem. Pick an integer $i$ such that
$$
(r + i)^{r} < k < (r + i + 1)^{r}.
$$
Since $(r + \log k + 1)^{r + 1} > k$ but
$(r + \log k)^{r} < k$ it follows that $i \asymp \log k$.
In particular $(r + i)^{r} \sim k$ as $k \rightarrow \infty$.
Set $\ell = r + i$ and let $\mathcal{N} = \{n_1 , \ldots, n_k\}$ be
a set containing the $(r + i)^{r} < k$ divisors of $\mathcal{P}(r,\ell)$
and $k - (r + i)^{r}$ other integers picked at random.
 Then according to G\'al's identity, highlighted above, the GCD sum
$$
\frac{1}{k}
\sum_{\substack{i, j \leq k \\ n_i, n_j \in \mathcal{N}}} \frac{(n_i, n_j)^2}{n_i n_j}
$$
is at least,
\begin{align*}
&\geq (r+i)^r \prod_{p < p_r} \Big ( 1 + 2 \sum_{v = 1}^{r + i - 1} \frac{1}{p^v}
\cdot \Big ( 1 - \frac{v}{r + i} \Big ) \Big ) \\ & > (1 + o(1)) k
\prod_{p < p_r} \Big (1 - \frac{1}{p^2} \Big )^{-2}
\times
\prod_{p < p_r} \big (1 + 2 \sum_{v = 1}^{r + i - 1}
\frac{1}{p^v} \cdot\Big ( 1 - \frac{v}{r + i} \Big )
\Big ) \Big (1 - \frac{1}{p} \Big )^2
\end{align*}
By Merten's theorem the first product is asymptotically
equal to $(e^{\gamma} \log p_r)^2 \sim (e^\gamma \log \log k)^2$
as $k \rightarrow \infty$. On the other the second product converges
as $k \rightarrow \infty$ to
$$
\prod_{p} \Big (1 + 2 \sum_{v = 1}^{\infty} \frac{1}{p^v}
\Big ) \Big ( 1 - \frac{1}{p} \Big )^2 = \frac{6}{\pi^2}.
$$
Combining these two observations the claim follows.

\section{Carleson-Hunt bounded variation}
We now turn our attention to the proof of Theorem \ref{thm:BVmax}. Our argument will depend on the Carleson-Hunt theorem \cite{CarlesonHunt}, stated below.
\begin{proposition}
There exists an absolute constant $c > 0$ such that
$$
\int_{\mathbb{T}} \Bigg ( \max_{1 \leq M \leq N} \Bigg | \sum_{k = 1}^{M}
c(k) e(k x) \Bigg | \Bigg )^2 dx \leq c \sum_{k = 1}^{N} |c(k)|^2
$$
for any finite sequence $(c(k))$.
\end{proposition}
We start by writing $f$ as a Fourier series
$$f(x) = \sum_{j \in \mathbb{Z} } a(j) e(jx),$$
where we have the inequality $|a(j)| \ll (1+|j|)^{-1}$. Next we split the Fourier series of $f$ into two parts based on the factorization of the $j$. Here $A$ denotes a large real constant to be specified later and $P^{+}(j)$ corresponds to the largest prime factor of $|j|$. Let
$$r(x) := \sum_{P^{+}(j) > (\log(N))^{2A+2} } a(j) e(jx), \hspace{1cm} p(x) := \sum_{P^{+}(j) \leq   (\log(N))^{2A+2} } a(j) e(jx)$$
so that $f(x) = p(x) + r(x)$. It suffices to prove, for $g(x) \in \{p(x),r(x)\}$ the inequality:
\begin{equation}\label{eq:maxP} \int_{\mathbb{T}} \left( \max_{1\leq M \leq N} \left| \sum_{k=1}^{N} c(k) g(n_k x) \right|  \right)^2 dx \ll   \left( \log\log (N) \right)^2 \sum_{k=1}^{N} |c(k)|^2.
\end{equation}
For $g(x) =p(x)$ we may write the square root of (\ref{eq:maxP}) as
$$\left( \int_{\mathbb{T}} \left( \max_{1\leq M \leq N} \left| \sum_{k=1}^{N} \sum_{P^{+}(j) \leq   (\log(N))^{2A+2} }  c(k) a(j) e(j n_k x) \right|  \right)^2 dx\right)^{1/2}  $$
$$\ll  \sum_{P^{+}(j) \leq   (\log(N))^{2A+2} }  |a(j)| \left( \int_{\mathbb{T}} \left( \max_{1\leq M \leq N} \left| \sum_{k=1}^{N} c(k)  e(j n_k x) \right|  \right)^2 dx \right)^{1/2}.$$
Applying the classical Carleson-Hunt inequality this is bounded by
$$ \ll  \left(\sum_{k=1}^{N} |c(k)|^2 \right)^{1/2}  \sum_{P^{+}(j) \leq   (\log(N))^{2A+2} }  |a(j)|$$
and it remains to notice that
\begin{align*}
\sum_{P^{+}(j) \leq   (\log(N))^{2A+2} } |a(j)| & \ll \sum_{P^{+}(j) \leq   (\log(N))^{2A+2} } j^{-1} \\ & = \prod_{p \leq  (\log(N))^{2A+2} } \frac{1}{ 1 - p^{-1}} \ll \log\log N
\end{align*}
by Merten's theorem. This completes the analysis of (\ref{eq:maxP}).

We now consider the left side of (\ref{eq:maxP}) with $g(x)=r(x)$. The key ingredient in the analysis of (\ref{eq:maxP}) will be the following almost orthogonality property of the functions $r(n_k x)$, which will be proved shortly.
\begin{lemma}\label{lem:orthoR}With the notation and conditions stated above, if $I=[M_1,M_2] \subseteq [1,N]$ then
$$ \int_{\mathbb{T}} \left( \sum_{k \in I} c(k) r(n_k x) \right)^2 dx \ll \frac{\left( \log \log  N  \right)^2 }{ \left( \log  N  \right)^{2A}} \sum_{k \in I} |c(k)|^2 . $$
\end{lemma}
Assuming Lemma \ref{lem:orthoR} for the moment, we may deduce a maximal version of this inequality at the expense of an additional factor of $\log N$ using a  Radamacher-Menshov-type argument inequality (see \cite{LewkoLewko} for a systematic discussion of this technique). This lemma will then imply (\ref{eq:maxP}) with $g(x)=r(x)$ for fixed $A>1$.
\begin{lemma}With the notation and conditions stated above,
$$\int_{\mathbb{T}} \left( \max_{1\leq M \leq N} \left| \sum_{k=1}^{M} c(k) r(n_k x) \right|  \right)^2 dx \ll \log^2 N  \cdot \frac{ \left( \log\log  N  \right)^2}{\log^{2A} N } \sum_{k=1}^{N} |c(k)|^2.$$
\end{lemma}
\begin{proof}Without loss of generality assume that $N = 2^n$ is a power of $2$. We consider the set of diadic subintervals of $[1,N]$:
$$\mathbb{D} := \{ [2^\ell (m-1), 2^{\ell-1}m) : 0 \leq \ell \leq n,  0 \leq m\leq 2^{n-\ell}  \}. $$
Let $1 \leq t(x) \leq N$ denote the length of the maximal partial sum at $x$. We may write the interval $[1,t(x)]$ as a disjoint union of at most $O(\log(N))$ diadic intervals (elements of $\mathbb{D}$). It follows, for fixed $x$, that for any $t(x) \leq N$ there exists a disjoint decomposition of
$[1,t(x)]$ into a union of $O(\log N)$ elements $\{\mathcal{D}_s^{(x)}\}_{s=1}^{\log N }$. Here $\mathcal{D}_s^{(x)}$ are disjoint dyadic intervals depending on $x$.  Hence
$$ \left| \sum_{k=1}^{t} c(k) r(n_k x) \right|^2 \leq \left( \sum_{s=1}^{\log N } \left| \sum_{k \in \mathcal{D}_s^{(x)} } c(k) r(n_k x) \right| \right)^2
\leq \log N \sum_{s=1}^{\log N}  \left| \sum_{k \in \mathcal{D}_s } c(k) r(n_k x) \right|^2 .$$
Summing over all dyadic intervals the dependence on $x$ may be removed. Indeed we have
$$\int_{\mathbb{T}} \left( \max_{1\leq M \leq N} \left| \sum_{k=1}^{N} c(k) r(n_k x) \right|  \right)^2  \ll \int_{\mathbb{T}}\log N  \sum_{\mathcal{D} \in \mathbb{D}} |\sum_{k \in \mathcal{D}}c(k) r(n_k x) |^2 $$
$$\ll \log N \sum_{\mathcal{D} \in \mathbb{D}} \int_{\mathbb{T}}|\sum_{k \in \mathcal{D}}c(k) r(n_k x) |^2dx.$$
Finally, Lemma \ref{lem:orthoR} combined with the observation that each integer $k$ occurs in $O(\log N )$ diadic intervals, allows us to bound this as
$$\ll \log^2 N  \frac{(\log\log N )^2}{\log^{2A}(N)} \sum_{k=1}^{N}|c(k) |^2.$$
This completes the proof.
\end{proof}
We need a simple majorant principle,
stated below.
\begin{lemma}
Let $c_k \geq 0$ be a non-negative sequence.
Given a sequence $a_n \geq 0$ and another sequence
$b_n \geq 0$ such that $a_n \leq b_n$ we have,
$$
\mathbb{E} \Bigg [ \Big | \sum_{n} a_n X(n) \Big |^2
\cdot \Big | \sum_{n} c_n X(n) \Big |^2 \Bigg ]
\leq
\mathbb{E}
\Bigg [ \Big | \sum_{n} b_n X(n) \Big |^2
\cdot \Big | \sum_{n} c_n X(n) \Big |^2 \Bigg ].
$$
\end{lemma}
\begin{proof}
By non-negativity of $c_k, a_n$ and $a_n \leq b_n$,
\begin{align*}
& \mathbb{E} \Bigg [ \Big | \sum_{n} a_n X(n) \Big |^2 \cdot
\Big | \sum_{n} c_n X(n) \Big |^2 \Bigg ] = \sum_{n} \Big ( \sum_{n = k\ell}
a_k c_{\ell} \Big )^2 \\ & \leq
\sum_{n} \Big ( \sum_{n = k \ell} b_k c_{\ell} \Big )^2
= \mathbb{E} \Bigg [ \Big | \sum_{n} b_n X(n) \Big |^2 \cdot
\Big | \sum_{n} c_n X(n) \Big |^2 \Bigg ]
\end{align*}
as claimed.
\end{proof}
We will also need a simple moment calculation.
\begin{lemma}
For any integer $\ell \geq 1$, and any $y \leq z$, and $\sigma > 0$,
$$
\mathbb{E} \Bigg [ \Big | \sum_{y \leq p \leq z} \frac{ X(p)}{p^{\sigma}} \Big |^{2\ell}
\Bigg ] \ll \ell! \cdot \Big ( \sum_{y \leq p \leq z} \frac{1}{p^{2\sigma}} \Big )^{\ell}.
$$
\end{lemma}
\begin{proof}
See \cite[Lemma 3.2]{LLR}.
\end{proof}

We are now ready to prove Lemma 10.

\begin{proof}[Proof of Lemma 10]
Let $J = (\log N)^{2A + 2}$.
Expanding $r$ and using orthogonality we get
\begin{align*}
\int_{0}^{1} \Bigg ( \sum_{k = M_1 + 1}^{M_2}
c(k) r(n_k x) \Bigg )^2 dx
= &  \frac{1}{2} \sum_{k_1,k_2 \in I} \sum_{\substack{P^{+}(j_1) > J \\
P^{+}(j_2) > J\\j_1 n_{k_1} = j_2 n_{k_2} }} c(k_1)c(k_2) a(j_1)a(j_2)
\\ \ll & \sum_{k_1, k_2 \in I}
\sum_{\substack{P^{+}(j_1) > J \\ P^{+}(j_2) > J \\ j_1 n_{k_1} = j_2 n_{k_2}}}
|c(k_1)|\cdot |c(k_2)| \cdot \frac{1}{j_1 j_2}
\end{align*}
since $|a(n)| \ll n^{-1}$.
The condition $j_1 n_{k_1} = j_2 n_{k_2}$ can be
expressed as
$$\mathbb{E} [X(j_1)X(n_{k_1})\overline{X(j_2)X(n_{k_2})}]$$
and therefore we re-write the previous sum as
$$
\mathbb{E} \Bigg [ \Big | \sum_{P^{+}(n) > J} \frac{X(n)}{n} \Big |^2
\cdot \Big | \sum_{n} |c(n)| X(n) \Big |^2 \Bigg ].
$$
If $P^+(n) > J$ then $n$ can be written as $n = p m$ with $p > J$ a prime.
Therefore by the majorant principle the above expression is less than or equal to
\begin{equation} \label{tobound}
\mathbb{E} \Bigg [ \Big | \sum_{p > J} \frac{X(p)}{p}
\cdot \zeta(1,X) \Big |^2 \cdot \Big | \sum_{n} |c(n)| X(n) \Big |^2
\Bigg ].
\end{equation}
Proceeding as in our proof of G\'al's theorem, for any choice of positive
integer $\ell > 0$ we have,
\begin{equation} \label{bounder}
\Big | \sum_{p > J} \frac{X(p)}{p} \Big |^2 < \frac{1}{(\log N)^{2A}} +
(\log N)^{2A\ell} \cdot \Big | \sum_{p > J} \frac{X(p)}{p} \Big |^{2(\ell + 1)}.
\end{equation}
The contribution of $(\log N)^{-2A}$ to (\ref{tobound}) is
$$
\frac{1}{(\log N)^{2A}} \cdot \mathbb{E} \Big [ |\zeta(1, X)|^2
\cdot \Big | \sum_{n} |c(n)| X(n) \Big |^2 \Big ]
\ll \frac{(\log\log N)^{2}}{(\log N)^{2A}} \cdot \sum_{n} |c(n)|^2
$$
as is seen by following the same steps as in our proof of G\'al's theorem. On the other hand the contribution
of $(\log N)^{2A\ell} \cdot |\sum_{p > J} X(p) / p|^{2(\ell+1)}$ to
(\ref{tobound}) is bounded by
\begin{align*}
\ll & (\log N)^{2A\ell} \cdot \mathbb{E} \Bigg [ \Big | \sum_{p > J} \frac{X(p)}{p}
\Big |^{2(\ell+1)} \cdot |\zeta(1, X)|^2 \cdot \Big | \sum_{n} |c(n)| X(n)
\Big |^2 \Bigg ] \\
\leq & (\log N)^{2A\ell} \cdot N \Big ( \sum_{n} |c(n)|^2 \Big ) \cdot
\mathbb{E} \Bigg [ \Big | \sum_{p > J} \frac{X(p)}{p}
\Big |^{2(\ell+1)} \cdot |\zeta(1, X)|^2 \Bigg ].
\end{align*}
By Cauchy-Schwarz and Lemma 13 the expectation is less than
\begin{align*}
\mathbb{E} \Big [ |\zeta(1, X)|^4 \Big ]^{1/2}
\cdot & \mathbb{E} \Bigg [ \Big | \sum_{p > J } \frac{X(p)}{p}
\Bigg |^{4(\ell+1)}  \Bigg ]^{1/2} \ll \\ & \ll 1 \cdot \Big ( (2(\ell+1))! \cdot \Big ( \sum_{p > J} \frac{1}{p^2}
\Big )^{2(\ell+1)} \Big )^{1/2} \ll \Big ( \frac{4\ell}{J} \Big )^{\ell}.
\end{align*}
It follows that the total contribution obtained by inserting (\ref{bounder}) into (\ref{tobound}) is
$$
\ll \frac{(\log\log N)^{2}}{(\log N)^{2A}} \sum_{n} |c(n)|^2
+ \Big ( \frac{4\ell (\log N)^{2A}}{J} \Big )^{\ell}
\cdot N \sum_{n} |c(n)|^2.
$$
Since $J = (\log N)^{2A + 2}$ choosing $\ell = \log N$ we see that the second
term is negligible compared to the first, and we've obtained the
desired bound.

\end{proof}

\section{Proof of Corollary \ref{cor:AE} and Corollary \ref{cor:EK}}

Corollary \ref{cor:AE} and Corollary \ref{cor:EK} follow immediately from the improved
Carleson-Hunt inequality established in Theorem \ref{thm:BVmax}. We refer the
readers to \cite[Proof of Theorem 2 and 3]{ABS} for the details of this
standard
deduction.

\section{Discussion of the connection with the maximal size of $\zeta(s)$}

Our bound for G\'al's theorem essentially corresponds to
$
(1/\zeta(2\sigma)) \cdot e^{2V(\sigma)}
$
where $V(\sigma)$ is the largest $V$ such that
$
\mathbb{P}(\log |\zeta(\sigma, X)| > V) < k^{-1-\varepsilon}.
$
We expect this $V$ to coincide with
$$
\sup_{|t| < k} \log |\zeta(\sigma + it)|.
$$
The reason for this is the following: if the maximum value $M$ of $|\zeta(\sigma + it)|$, $|t| < k$, is attained at $t = t_0$, then we have
$$
|\zeta(\sigma + it_0 + i\varepsilon)| > M/2
$$
for $|\varepsilon| \ll 1 /\log k$ (see \cite{GonekFarmer}). Thus if the probabilistic model
predicts that $|\zeta(\sigma + it)| > e^{V}$ on a set of measure at most $k^{-\varepsilon}$ for $|t| < k$,  then we expect
$|\zeta(\sigma + it)| < e^{V}$ to hold for all $|t| < k$.

We do not expect our approach to deliver sharp estimates for the constants
$C(\alpha)$ as defined in Theorem \ref{thm:spec}. However, for future reference we note
that it seems that with more work one can show that
for $1/2 < \alpha < 1$ we have
$$C(\alpha) = G_1(\alpha) \alpha^{-2\alpha} (1 - \alpha)^{\alpha - 1} + o(1)$$
as $k \rightarrow \infty$,
with $$G_1(\alpha) = \int_{0}^{\infty} \log I_0(u) u^{-1-1/\sigma} du$$
and $$I_0(u) = \sum_{n = 0}^{\infty} (u/2)^{2n}/ n!^2$$ the modified Bessel function of order 0. This is suggested by Lamzouri's conjecture (see Remark 2 after Theorem 1.5 in \cite{Youness})
as to the
order of
$$
\sup_{|t| < k} \log |\zeta(\sigma + it)|.
$$
His conjecture is made through an analysis of the random model
$\zeta(\sigma, X)$, as described above (see also \cite{GonekFarmer}).

\section{Acknowledgements}
The first author would like to thank Jeff Vaaler for first bringing G\'al's theorem to his attention and many discussions on it and related topics.

\bibliography{GCD}
\bibliographystyle{plain}

\end{document}